\documentclass[12pt]{article}
\usepackage{amsmath}
\usepackage{amssymb}
\usepackage{amsthm}
\usepackage{mathabx}
\usepackage{caption}
\usepackage[usenames]{color}
\usepackage{amscd}
\usepackage{dsfont}
\usepackage{indentfirst}

\usepackage[colorlinks=true,linkcolor=blue,filecolor=red,
citecolor=webgreen]{hyperref}
\definecolor{webgreen}{rgb}{0,.5,0}

\def\C{{\mathds{C}}}
\def\Q{{\mathds{Q}}}
\def\R{{\mathbb{R}}}
\def\N{{\mathds{N}}}
\def\Z{{\mathds{Z}}}
\def\1{{\bf 1}}

\def\pont{$\bullet$ }

\newtheorem{theorem}{Theorem}

\newtheorem{lemma}[theorem]{Lemma}
\newtheorem{corollary}[theorem]{Corollary}

\begin{document}
	
	\title{{\bf Counting $k$-tuples of positive integers
			such that the values of several polynomials of $k$ variables are relatively
			prime}}
	\author{L\'aszl\'o T\'oth}
	\date{}
	\maketitle
	
	\begin{abstract} Let $F=(f_1(x_1,\ldots,x_k),\ldots, f_m(x_1,\ldots,x_k))$ be a system of nonconstant polynomials 
		of $k$ variables with integer coefficients and let
		\[ {\gcd}_F(x_1,\ldots,x_k)= \gcd(f_1(x_1,\ldots,x_k),\ldots, f_m(x_1,\ldots,x_k)).
		\]
		We obtain an unconditional asymptotic formula  for the sum
		\[
		\sum_{1\le x_1,\ldots,x_k\le x} h({\gcd}_F(x_1,\ldots,x_k)),
		\] 
		where $F$ is a system of $m\ge 2$ polynomials of $k\ge 2$ variables subject to certain general properties, and $h$ is a bounded strongly multiplicative function. In particular, we deduce an asymptotic formula with error term concerning the density of $k$-tuples of positive integers $(x_1,\ldots,x_k)$ such that ${\gcd}_F(x_1,\ldots,x_k)=1$, given by the Ekedahl-Poonen formula. 
	\end{abstract}
	
	{\sl 2020 Mathematics Subject Classification}: Primary 11N37, Secondary 11C08, 11A05
	
	{\sl Key Words and Phrases}: integer polynomials of several variables, systems of congruences, greatest common divisor, 
	asymptotic density, asymptotic formula
	
	\section{Introduction}
	
	Let $F=(f_1(x_1,\ldots,x_k),\ldots, f_m(x_1,\ldots,x_k))$ be a system of nonconstant polynomials 
	of $k$ variables with integer coefficients, where $k,m\in \N:=\{1,2,\ldots\}$. 
	Let $N_F(n)$ denote the number of solutions of the system of congruences $f_i(x_1,\ldots,x_k)\equiv 0$ (mod $n$) ($1\le i\le m$), that is, the number of solutions $(x_1,\ldots,x_k)\in \{1,\ldots,n\}^k$. The function $n\mapsto N_F(n)$ is multiplicative. 
	
	Also, for $x_1,\ldots,x_k\in \N$ let
	\begin{equation*}
		{\gcd}_F(x_1,\ldots,x_k)= \gcd(f_1(x_1,\ldots,x_k),\ldots, f_m(x_1,\ldots,x_k)).
	\end{equation*} 
	
	It is known that if $k\ge 1$, $m\ge 2$ and the polynomials $f_i$ ($1\le i\le m$) are relatively prime (over the field $\Q$), 
	then the asymptotic density of the $k$-tuples of positive integers $(x_1,\ldots,x_k)$ such that ${\gcd}_F(x_1,\ldots,x_k)=1$ is 
	\begin{equation} \label{C_F}
		C_F= \prod_p \left(1- \frac{N_F(p)}{p^k} \right),    
	\end{equation}
	the product being over the primes $p$. This is the Ekedahl-Poonen formula, initially given in the case $m=2$ and then extended for $m\ge 2$. See Poonen \cite[Th.\ 3.1]{Poo2003}, Bodin and D\`{e}bes \cite[Th.\ 1.5]{BodDeb2023}.
	
	In the case of arbitrary nonconstant polynomials of one variable with integer coefficients, where $f_i(x_i)$ only depends on $x_i$ ($k=m\ge 2$, $1\le i\le k$)
	one has the asymptotic formula
	\begin{equation} \label{form_1_var}
		\sum_{\substack{1\le x_1,\ldots,x_k\le x\\ \gcd(f_1(x_1),\ldots, f_k(x_k))=1}} 1 = 
		x^k \prod_p \left(1-\frac{\varrho(p)}{p^k}\right) + O(x^{k-1} (\log x)^c),
	\end{equation}
	with some positive constant $c$, where $\varrho(p)=\prod_{1\le i\le k} \varrho_i(p)$ and 
	$\varrho_i(p)$ denotes the number of solutions of the congruence $f_i(x)\equiv 0$ (mod $p$).
	A similar formula holds with the slightly more general condition $\gcd(f_1(x_1),\ldots, f_k(x_k))=\ell$, where $\ell \in \N$ is fixed. See Chidambaraswamy and Sitaramachandrarao \cite{ChiSit1985, ChiSit1987}. 
	
	The goal of this paper is to generalize formula \eqref{form_1_var} for polynomials of several variables that satisfy certain conditions.
	More generally, we consider the sum
	\begin{equation}  \label{S_F_h}
		S_{F,h}(x)= \sum_{1\le x_1,\ldots,x_k\le x} h({\gcd}_F(x_1,\ldots,x_k)),
	\end{equation}
	where $h:\N \to \R$ is an arithmetic function. 
	As far as we know, there are no such asymptotic formulas with error terms in the literature. 
	
	It is a difficult task to deduce asymptotic formulas for sums given by \eqref{S_F_h}. For example, if $h=\mu^2$, then  
	a result is only known in the case of one polynomial of one variable ($k=m=1$)
	Namely, Murty and Pasten \cite[Th.\ 1.1]{MP2014} proved the following result assuming the ABC conjecture. Let $f$ be a polynomial with integer
	coefficients, of degree $r\ge  2$, without repeated factors, and with $\gcd (f(n):n\in \N)$ squarefree. Then
	\begin{equation*}
		\sum_{n\le x} \mu^2(f(n)) =c_f x+ O\left(\frac{x}{(\log x)^{\gamma}}\right),
	\end{equation*}
	where $\gamma>0$ is a computable constant that only depends on $r$ (not on the particular
	$f$), and
	\begin{equation*}
		c_f=\prod_p \left(1-\frac{\varrho_f(p^2)}{p^2} \right),
	\end{equation*}
	where $\varrho_f(p^2)$ is the number of solutions of the congruence $f(x)\equiv 0$ 
	(mod $p^2$). Also see Granville \cite{Gra1998}, Poonen \cite{Poo2003}, and the recent preprint by 
	Sofos \cite{Sof2026}.
	
	However, we are able to deduce an unconditional asymptotic formula for \eqref{S_F_h} if $F$ is a system of $m\ge 2$ polynomials of $k\ge 2$ variables, subject to certain general properties, and $h$ is a bounded strongly multiplicative function. The essential property of strongly multiplicative functions $h$, which we use in the proof, is that $(\mu*h)(p^a)=0$ for all prime powers $p^a$ with $a\ge 2$. 
	Therefore, it is enough to apply known upper bounds for $N_F(p)$, where $p$ is a prime and $\Z/p\Z$ is a field. 
	In particular, in the case $h(1)=1$ and $h(n)=0$ for $n>1$, we deduce a direct generalization of formula \eqref{form_1_var} 
	for polynomials of several variables. See Corollary \ref{Corollary_dens}.
	
	In the case of some other functions $h$, we would also need good upper bounds for $N_F(p^a)$, with certain values $a\ge 2$, which are not known in the literature. Here, the main difficulty is that for $a\ge 2$, $\Z/p^a\Z$ is not an integral domain, and even in the case of one variable, the number $N_F(p^a)$ of solutions can be large. For $h=\mu^2$ one has $(\mu*\mu^2)(p^2)=-1$, $(\mu*\mu^2)(p^a)=0$ for $a\in \{1,3,4,\ldots\}$, and an upper bound for $N_F(p^2)$ would be needed. 
	
	Our main general results are formulated in Section \ref{Section_Main_results}.
	Special cases are discussed in Section \ref{Section_Special_cases}. Certain lemmas needed for the proofs are given in Section \ref{Section_Preliminaries}. The proofs of the main results, using only elementary arguments, are presented in Section \ref{Section_Proofs}. 
	
	Throughout the paper, we use the following notation: $\mu$ is the M\"obius function, $\omega(n)$ and $\kappa(n)$ denote 
	the number and the product of distinct prime divisors of $n\in \N$, respectively; $\delta(n)=\lfloor 1/n\rfloor$ ($n\in \N$), $*$ denotes the convolution of arithmetic functions, $\sum_p$ and $\prod_p$ represent sums and products over the primes, respectively.
	
	\section{General results} \label{Section_Main_results}
	
	We say that a system $F$ of nonconstant polynomials $f_i(x_1,\ldots,x_k)$ of total degree $d_i$,
	has the positivity property P, if all coefficients of the monomials of degree $d_i$ are positive \textup{($1\le i\le m$)}.
	
	Our first result is the following. 
	
	\begin{theorem} \label{Th_1}
		Let $k,m\ge 2$ and let $F$ be a system of polynomials with property P. 
		Let $h:\N\to \C$ be a bounded strongly multiplicative function.
		
		Assume that for all sufficiently large primes $p$, $N_F(p)\le C p^{k-2}$ and 
		$|h(p)|\le D$, where $C>0$, $D\ge 0$ are constants.
		
		Then
		\begin{equation} \label{est_h}
			S_{F,h}(x)= x^k \prod_p \left(1+ \frac{(h(p)-1)N_F(p)}{p^k} \right) + O\left(x^{k-1}(\log x)^{C(D+1)}\right).
		\end{equation}
		
		More precisely, if the system $F$ has property P and the polynomials are relatively prime, 
		then \eqref{est_h} holds with $C=d^m$, where $d=\max_{1\le i\le m} d_i$.
	\end{theorem}
	
	\begin{corollary} \textup{($h=\delta$)} \label{Corollary_dens} 
		Let $k,m\ge 2$. Assume that the system $F$ 
		has the property P and the polynomials of $F$ are relatively prime. 
		Then
		\begin{equation} \label{form_k_var}
			\sum_{\substack{1\le x_1,\ldots,x_k\le x\\ \gcd_F(x_1,\ldots,x_k)=1}} 1 = 
			C_F x^k  + O\left(x^{k-1}(\log x)^{d^m}\right),
		\end{equation}
		where $C_F$ is defined by \eqref{C_F}.
	\end{corollary}
	
	Compare formula \eqref{form_k_var} to \eqref{form_1_var}. Note that in our recent paper \cite{CsiTot2026}
	we have investigated the generalized Euler function $\Phi_F(n)$, defined as  
	\begin{equation*} 
		\Phi_F(n)= \# \{ (x_1,\ldots,x_k)\in \{1,\ldots,n\}^k: \gcd({\gcd}_F(x_1,\ldots, x_k),n) =1\},
	\end{equation*}
	and obtained the following result. Assume that the system of polynomials $F$ is such that for every $1\le i\le m$ the coefficients of $f_i(x_1,\ldots,x_k)$ are relatively prime. Then
	\begin{align*} 
		\sum_{n\le x} \Phi_F(n) = x^{k+1} \prod_p \left(1-\frac{N_F(p)}{p^{k+1}} \right) +
		O(x^{k+1-1/t+\varepsilon}),
	\end{align*}
	where $t=\min_{1\le i\le m} d_i$ and $\varepsilon > 0$ is arbitrary. See \cite[Th.\ 2.3]{CsiTot2026}.
	
	\section{Special cases} \label{Section_Special_cases}
	
	We consider the following special cases of polynomials. First, let $k\ge 2$, $m=2$, $f_1(x_1,\ldots,x_k)=x_1\cdots x_k$ and $f_2(x_1,\ldots,x_k)=x_1+\cdots +x_k$.
	We have the next result.
	
	\begin{corollary} \label{Corollary_spec_h}
		Let $h:\N\to \C$ be a bounded strongly multiplicative function and assume that for all sufficiently large primes $p$, $|h(p)|\le D$, where $D\ge 0$ is a constant. Then for every $k\ge 2$,
		\begin{equation*}
			\sum_{1\le x_1,\ldots,x_k\le x} h(\gcd(x_1\cdots x_k, x_1+\cdots +x_k))
		\end{equation*}
		\begin{equation*} 
			=  x^k \prod_p \left(1+ \frac{(h(p)-1)N_k(p)}{p^k} \right)  +O\left(x^{k-1} (\log x)^{k(D+1)}\right),
		\end{equation*}
		where 
		\begin{equation*}
			N_k(p)= p^{k-1}- \frac{p-1}{p}\left((p-1)^{k-1}+ (-1)^k\right).
		\end{equation*}
	\end{corollary}
	
	\begin{corollary} \textup{($h=\delta$)} \label{Corollary_spec_delta}
		For every $k\ge 2$,
		\begin{equation*}
			\sum_{\substack{1\le x_1,\ldots,x_k\le x\\ \gcd(x_1\cdots x_k, x_1+\cdots +x_k)=1}} 1 =
			C_k x^k +O\left(x^{k-1} (\log x)^k\right),
		\end{equation*}
		where 
		\begin{equation*}
			C_k = \prod_p \left( 1- \frac1{p} +\frac{(p-1)^k+ (-1)^k(p-1)}{p^{k+1}} \right).
		\end{equation*}
	\end{corollary}
	
	Note that $C_2=\prod_p (1-1/p^2)=6/\pi^2 \doteq 0.607927$ and $C_3=\prod_p (1-(3p-2)/p^3) \doteq 
	0.286747$ is the strongly carefree constant (it is the asymptotic density of ordered pairs $(x,y)\in 
	\N^2$ such that $x$ and $y$ are coprime and both are squarefree; $C_3$ is also the asymptotic density 
	of the ordered triplets $(x,y,z)\in \N^3$ such that $x,y,z$ are pairwise coprime; see \cite{OEIS}). 
	
	We also remark that the generalized Euler function 
	\begin{equation*}
		\varphi_k(n)= \# \{ (x_1,\ldots,x_k)\in \{1,\ldots,n\}^k: \gcd(x_1\cdots x_k,x_1+\cdots+x_k,n) =1\}
	\end{equation*}
	was investigated in our paper \cite{Tot2022}.  
	
	The asymptotic formulas of Theorem \ref{Th_1} and Corollary \ref{Corollary_spec_h} also hold for 
	the following bounded strongly multiplicative functions with the given constants $D$ (for every $\varepsilon>0$):
	
	\pont $h(n)= E^{\omega(n)}$ with $|E|\le 1$, $E\ne 0$ a real number, $D=|E|$;
	
	\pont $h(n)=1/\kappa(n)$, $D=\varepsilon$;
	
	\pont $h(n)=\prod_{p\mid n} \left(1+1/p^2 \right)$, where the series $\sum_p (h(p)-1)= \sum_p 1/p^2$ converges, $D=1+\varepsilon$;
	
	\noindent also see Section \ref{Section_Preliminaries} for some other examples.
	
	Now, let $k=m$ and select the system of linear polynomials $f_i(x)=a_{i1}x_1+\cdots +a_{ik}x_k+b_i$, where $a_{ij},b_i\in \Z$ ($1 \le i,j\le k$). 
	
	\begin{theorem} \label{Th_linear}
		Let $k\ge 2$ and assume that the matrix $M=(a_{ij})_{1\le i,j\le k}$ has positive integer entries and is unimodular, that is, $\det(M)=\pm 1$. 
		Let $b_i \in \Z$ \textup{($1\le i\le k$)} and let $h$ be an arbitrary bounded function. Then
		\begin{equation*}
			\sum_{1\le x_1,\ldots,x_k\le x} h(\gcd(a_{11}x_1+\cdots +a_{1k}x_k+b_1,\ldots,
			a_{k1}x_1+\cdots +a_{kk}x_k+b_k)) 
		\end{equation*}
		\begin{equation} \label{sum_h}
			= \frac{x^k}{\zeta(k)}\sum_{n=1}^{\infty}\frac{h(n)}{n^k} + O(R_k(x)),
		\end{equation}
		where $R_k(x)=x^{k-1}$ for $k\ge 3$ and $R_2(x)=x(\log x)^2$.
	\end{theorem} 
	
	Note that if $h=\mu^2$, then the constant of the main term in \eqref{sum_h} is $1/\zeta(2k)$.
	
	\begin{corollary} \textup{($h=\delta$)}  \label{Corollary_linear}
		In the conditions of Theorem \ref{Th_linear} for the linear polynomials $f_i$ \textup{($1\le i\le k$)}
		we have
		\begin{equation*}
			\sum_{\substack{1\le x_1,\ldots,x_k\le x\\ \gcd(a_{11}x_1+\cdots +a_{1k}x_k+b_1,\ldots,
					a_{k1}x_1+\cdots +a_{kk}x_k+b_k)=1}} 1
			= \frac{x^k}{\zeta(k)} + O(T_k(x)),
		\end{equation*}
		where $T_k(x)=x^{k-1}$ for $k\ge 3$ and $T_2(x)=x \log x$.
	\end{corollary} 
	
	\section{Preliminaries} \label{Section_Preliminaries}
	
	Let $h:\N \to \R$ be an arithmetic function which is not identical zero. It is called multiplicative 
	if $h(mn)=h(m)h(n)$ for all $m,n\in \N$ with $\gcd(m,n)=1$. The function $h$ is strongly 
	multiplicative if $h$ is multiplicative and $h(p^a)=h(p)$ for every prime power $p^a$ ($a\ge 1$).
	Then $h(1)=1$ and $h(n)=\prod_{p\mid n} h(p)$ for all $n\in \N, n>1$. Also, 
	$(\mu*h)(p)=h(p)-1$ and $(\mu*h)(p^a)=h(p^a)-h(p^{a-1})= 0$ for every $a\ge 2$.
	
	It is easy to show that a strongly multiplicative function $h:\N \to \R$ is bounded if and only if the infinite product 
	$\prod_p \max(1,|h(p)|-1)$ converges. That is, $|h(p)|\le 1$ for all but finitely many primes $p$, or $|h(p)|>1$ for infinitely many primes and 
	the series $\sum_p \max(0, |h(p)|-1)$ converges. In the second case, for all but finitely many primes one has $|h(p)|\le 1+\varepsilon$ for every $\varepsilon>0$.   
	
	Examples of bounded strongly multiplicative functions $h$ are those given in Section \ref{Section_Main_results} and also the next functions:
	
	\pont $h(p)=1$ for $p\in S$, $h(p)=0$ for $p\notin S$, where $S$ is an arbitrary fixed subset of the primes;
	
	\pont $h(p)=p$ for finitely many primes $p$, and $h(p)=1$ otherwise.
	
	\begin{lemma} \label{Lemma_cong} Let $F=(f_1(x_1,\ldots,x_k),\ldots, f_m(x_1,\ldots,x_k))$ be an arbitrary system of nonconstant polynomials with integer coefficients, where $k\ge 1$, $m\ge 2$. Let $y\in \N$ be fixed and $x>y$ be a real number.  Let $N_F(x,y,n)$ denote the number of solutions $(x_1,\ldots,x_k)$ of the system of congruences $f_i(x_1,\ldots,x_k)\equiv 0$ \textup{(mod $n$)} \textup{($1\le i\le m$)} such that $y+1\le x_1,\ldots,x_k\le x$.
		Then 
		\begin{equation} \label{N_F_x_y_d}
			N_F(x,y,n)= \left(\frac{x}{n}\right)^k N_F(n) + O\left( \left(\frac{x}{n}\right)^{k-1} N_F^*(n)\right),
		\end{equation}
		as $x\to \infty$, with $N_F^*(n)=N_F(n)$ if $N_F(n)\ge 1$ and $N_F^*(n)=1$ if $N_F(n)=0$.
	\end{lemma}
	
	\begin{proof} Let $q=\lfloor (x-y)/n \rfloor$ (integer part). Let $D=[y+1,x]^k$ and
		$E=[y+1,y+nq]^k\subseteq D$. Here $E$ is the union of a number of $q^k$ of hypercubes, each of 
		side length $n$. That is,
		\begin{equation*}
			E= \bigcup_{1\le j_1,\ldots,j_k\le q} E_{j_1,\ldots,j_k},
		\end{equation*}
		where
		\begin{equation*}
			E_{j_1,\ldots,j_k}= [y+(j_1-1)n+1,y+j_1n]\times \cdots \times [y+(j_k-1)n+1,y+j_kn].
		\end{equation*}
		
		In each hypercube $E_{j_1,\ldots,j_k}$ the number of solutions is $N_F(n)$. On the other hand,
		$D=[y+1,x]^k$ can be covered by the union of a number of $(q+1)^k$ of hypercubes of 
		side length $n$, namely
		\begin{equation*}
			D\subset F = \bigcup_{1\le j_1,\ldots,j_k\le q+1} E_{j_1,\ldots,j_k}.
		\end{equation*}
		
		Also, $F\setminus E$ is the union of a number of $(q+1)^k-q^k\ll q^{k-1}$ such hypercubes. Hence,
		the number of solutions in $D$ is 
		\begin{equation*}
			N_F(x,y,n)= q^k N_F(n) + O(q^{k-1}N_F^*(n))
		\end{equation*}
		\begin{equation*}
			= \left\lfloor \frac{x-y}{n}\right\rfloor^k N_F(n)+ O\left(\left\lfloor \frac{x-y}{n}
			\right\rfloor^{k-1}N_F^*(n)\right),
		\end{equation*}
		which gives \eqref{N_F_x_y_d}.
	\end{proof}
	
	We also need the following results.
	
	\begin{lemma}{\rm (Bodin and D\`{e}bes \cite[Cor.\ 6.1]{BodDeb2023})} \label{Lemma_ineq}
		If $k,m\ge 2$ and $f_i(x_1,\ldots,x_k)$ \textup{($1\le i\le m$)} are relatively prime polynomials in $k$ variables, each of degree at most $d$, then for all sufficiently large primes $p$,
		\begin{equation*} 
			N_F(p)\le C p^{k-2},
		\end{equation*}
		with $C=d^m$.
	\end{lemma}
	
	\begin{lemma} \label{Lemma_omega} Let $z>0$ be a real number and $k\ge 2$ be an integer. Then
		\begin{equation*}
			\sum_{n\le x} \frac{z^{\omega(n)}}{n} = O\left( (\log x)^z\right),
		\end{equation*}
		\begin{equation*}
			\sum_{n>x} \frac{z^{\omega(n)}}{n^k} = O\left(\frac{(\log x)^{z-1}}{x^{k-1}} \right).
		\end{equation*}
	\end{lemma}
	
	\begin{proof} It is known that 
		\begin{equation*}
			\sum_{n\le x} z^{\omega(n)} = f(z) x (\log x)^{z-1} + O\left(x (\log x)^{z-2} \right),
		\end{equation*}
		with some constant $f(z)$, even in a more precise form, for complex values of $z$.
		See, e.g., Selberg \cite{Sel1954}, Tenenbaum \cite[p.\ 300, Ch.\ II. 6, Th.\ 6.1]{Ten2015}.
		Now, the given estimates follow by partial summation. 
	\end{proof}
	
	\section{Proofs} \label{Section_Proofs}
	
	\begin{proof}[Proof of Theorem {\rm \ref{Th_1}}]
		By the positivity property P, for the given polynomials $f_i(x_1,\ldots,x_k)$ there exist $y_i\in \N$ such that $f_i(x_1,\ldots,x_k)\ge 1$ if $x_1,\ldots,x_k\ge y_i$ ($1\le i\le m$). 
		Let $y_0=\max \{y_i: 1\le i\le m\}\in \N$. 
		
		We have for $x>y_0+1$,
		\begin{equation*}
			S_{F,h}(x)= \sum_{y_0< x_1,\ldots,x_k\le x} h({\gcd}_F(x_1,\ldots,x_k)) + 
			\sum_{\substack{1\le x_1,\ldots,x_k\le x\\ \text{ at least one $x_i\le y_0$ }}} h({\gcd}_F(x_1,\ldots,x_k)).
		\end{equation*}
		
		Here the second sum is $O(x^{k-1})$, since the function $h$ is bounded.
		
		Let $S'_{F,h}(x)$ denote the first sum, and let 
		\begin{equation*}
			M(x) =\max \{f_i(x_1,\ldots,x_k): y_0+1\le x_1,\ldots,x_k\le x, 1\le i\le m \},
		\end{equation*}
		where $M(x)\ge 1$, $M(x)$ is non decreasing, and $M(x)\to \infty$ as $x\to \infty$. Also, $Ax^t\le M(x)\le B x^t$ with $t=\max_{1\le i\le m} d_i\ge 1$ and $A,B>0$ certain constants (depending on the polynomials). 
		
		Using that $h(n)=\sum_{d\mid n} (\mu*h)(d)$ ($n\in \N$), the sum $S'_{F,h}(x)$ can be written as
		\begin{equation*}
			S'_{F,h}(x)= \sum_{y_0+1\le x_1,\ldots,x_k\le x} \ \sum_{1\le d\mid \gcd(f_1(x_1,\ldots,x_k),\ldots, f_m(x_1,\ldots,x_k))} (\mu*h)(d) 
		\end{equation*}
		\begin{equation*}
			= \sum_{y_0+1\le x_1,\ldots,x_k\le x} \ \sum_{\substack{1\le d\mid f_i(x_1,\ldots,x_k)\\ 1\le i\le m}} (\mu*h)(d)
		\end{equation*}
		\begin{equation*}
			= \sum_{1\le d\le M(x)} (\mu*h)(d) \sum_{\substack{y_0+1\le x_1,\ldots,x_k\le x \\ f_i(x_1,\ldots,x_k)\equiv 0 \text{ (mod $d$)}, \, 1\le i\le m}} 1. 
		\end{equation*}
		
		(Here $y_0$ is needed to avoid the zero values of the polynomials, in this way we have $1\le d\le M(x)$.
		Otherwise, if $d\mid 0$, then the values of $d$ are not bounded.)
		
		The inner sum of above is $N_F(x,y_0,d)$, and we deduce by Lemma \ref{Lemma_cong} that
		\begin{equation*}
			S'_{F,h}(x) = \sum_{1\le d\le M(x)} (\mu*h)(d) \left( (\frac{x}{d})^k N_F(d) + O\left( (\frac{x}{d})^{k-1} N_F^*(d)\right)\right)
		\end{equation*}
		\begin{equation} \label{sum_errors}
			= x^k \sum_{d=1}^{\infty} \frac{(\mu*h)(d)N_F(d)}{d^k} + R^{(1)}_F(x) + R^{(2)}_F(x),
		\end{equation}
		where
		\begin{equation*}
			R^{(1)}_F(x) \ll x^k \sum_{d>M(x)} \frac{|(\mu*h)(d)| N_F^*(d)}{d^k},
		\end{equation*}
		\begin{equation*}
			R^{(2)}_F(x) \ll x^{k-1} \sum_{d\le M(x)} \frac{|(\mu*h)(d)| N_F^*(d)}{d^{k-1}}.
		\end{equation*}
		
		In \eqref{sum_errors} the series in the main term is absolutely convergent. To show this, note that the 
		involved functions are multiplicative, hence
		the series can be expanded into the Euler product
		\begin{equation*} 
			\sum_{d=1}^{\infty} \frac{(\mu*h)(d)N_F(d)}{d^k} = \prod_p \left(1+ 
			\frac{(h(p)-1)N_F(p)}{p^k}\right).
		\end{equation*}
		
		Using the assumptions, for all but finitely many primes $p$, say $p\ge p_0$, we have 
		$N_F(p)\le Cp^{k-2}$ and $|h(p)|\le D$. Hence the above infinite product is absolutely convergent, since  
		\begin{equation*} 
			\prod_p \left(1+ \frac{|h(p)-1|N_F(p)}{p^k} \right) \ll \prod_{p\ge p_0} \left(1+ \frac{C(D+1)}{p^2} \right)<\infty.
		\end{equation*}
		
		Now we estimate the error terms. It follows from $N_F(p)\le Cp^{k-2}$ ($p\ge p_0$), by 
		multiplicativity, that for all squarefree $n\in \N$,
		\begin{equation*}
			N_F(n)=\prod_{p\mid n} N_F(p) \ll  C^{\omega(n)} n^{k-2}. 
		\end{equation*}
		
		Using that $|h(p)|\le D$ ($p\ge p_0$) and $(\mu*h)(p^a)=0$ ($a\ge 2$) we deduce that for every $n\in \N$,
		\begin{equation*}
			|(\mu*h)(n)| N_F(n) \le \prod_{p\mid n} |h(p)-1| N_F(p) \ll (C(D+1))^{\omega(n)} n^{k-2}.
		\end{equation*}
		
		Hence we have by Lemma \ref{Lemma_omega},
		\begin{equation*}
			R^{(1)}_F(x)\ll x^k \sum_{d>M(x)} \frac{(C(D+1))^{\omega(d)}}{d^2} 
			\ll x^k \frac{(\log M(x))^{C(D+1)-1}}{M(x)}
		\end{equation*}
		\begin{equation*}
			\le x^k \frac{(\log Bx^t)^{C(D+1)-1}}{Ax^t}\ll x^{k-t}  (\log x)^{C(D+1)-1}, 
		\end{equation*}
		where $t\ge 1$. Also,
		\begin{equation*}
			R^{(2)}_F(x) \ll x^{k-1} \sum_{d\le M(x)} \frac{(C(D+1))^{\omega(d)}}{d} 
			\ll x^{k-1} (\log M(x))^{C(D+1)}
		\end{equation*}
		\begin{equation*}
			\le x^{k-1} (\log B x^t)^{C(D+1)} \ll x^{k-1} (\log x)^{C(D+1)},
		\end{equation*}
		finishing the proof of \eqref{est_h}.
		
		Furthermore, Lemma \ref{Lemma_ineq} ensures that one can choose $C=d^m$.
	\end{proof}
	
	\begin{proof}[Proof of Corollary {\rm \ref{Corollary_dens}}] This is the special case $h(1)=1$, $h(p)=0$ for every prime $p$. Hence, one can select $D=0$. 
	\end{proof}
	
	\begin{proof}[Proof of Corollary {\rm \ref{Corollary_spec_h}}]
		The polynomials $f_1(x_1,\ldots,x_k)=x_1\cdots x_k$ and $f_2(x_1,\ldots,x_k)=x_1+\cdots +x_k$
		have positive coefficients, and are of degrees $d_1=k$, $d_2=1$. 
		For primes $p$ the number of solutions of the 
		system of congruences $x_1\cdots x_k\equiv 0$ (mod $p$), $x_1+\cdots +x_k\equiv 0$ (mod $p$) is 
		\begin{equation*}
			N_k(p)= p^{k-1}- \frac{p-1}{p}\left((p-1)^{k-1}+ (-1)^k \right),
		\end{equation*}
		which is a polynomial of degree $k-2$. See \cite[Lemma 5.6]{CsiTot2026}. Here 
		$N_2(p)= 1$, $N_3(p)= 3p-2$, $N_4(p)= 4p^2-6p+3$,  
		$N_5(p)= 5p^3-10p^2+10p-4$.
		
		We show that $N_k(p)<k p^{k-2}$ for every $p$ and every $k\ge 2$.
		This is true for $k=2$, and assume that $k\ge 3$. Then by Lagrange's mean value theorem,
		\begin{equation*}
			p^{k-1}-(p-1)^{k-1} < (k-1)p^{k-2},
		\end{equation*}
		and we deduce that 
		\begin{equation*} 
			N_k(p) < p^{k-1} + \frac{p-1}{p} \left((k-1)p^{k-2}- p^{k-1} + (-1)^{k-1} \right) 
		\end{equation*}
		\begin{equation*} 
			= kp^{k-2} - (k-1) p^{k-3} + (-1)^{k-1} \frac{p-1}{p} < k p^{k-2},
		\end{equation*}
		as we claimed. Hence we have $C=k$ and apply Theorem \ref{Th_1}.
	\end{proof}
	
	\begin{proof}[Proof of Corollary {\rm \ref{Corollary_spec_delta}}] Apply Corollary \ref{Corollary_spec_h} for $h=\delta$ with $D=0$.
	\end{proof}
	
	\begin{proof}[Proof of Theorem {\rm \ref{Th_linear}}]
		Consider the system of linear congruences 
		\begin{align*}
			a_{11}x_1+\cdots +a_{1k}x_k+b_1 \equiv & \ 0 \textup{ (mod $n$)} \\ \ldots & \ldots \\
			a_{k1}x_1+\cdots +a_{kk}x_k+b_k \equiv & \ 0 \textup{ (mod $n$)},
		\end{align*}
		where $a_{ij},b_i\in \Z$ \textup{($1 \le i,j\le k$)} and $n\in \N$. Let $M=(a_{ij})_{1\le i,j\le k}$ and assume that $\gcd(\det(M),n)=1$. Then the system has a unique solution (mod $n$). See, e.g., Rosen \cite[Sect.\ 4.5]{Ros2011}.
		Hence, if $\det(A)=\pm 1$, then $N_F(n)=1$ for every $n\in \N$. 
		
		If the function $h$ is bounded, then $|(\mu*h)(n)|\ll \sum_{d\mid n} |\mu(d)|= 2^{\omega(n)}$.
		
		Following the notation and arguments in the proof of Theorem \ref{Th_1}, we have $t=1$,
		$Ax\le M(x)\le Bx$ with some constants $A,B>0$. According to \eqref{sum_errors},
		\begin{equation} \label{sum_special}
			S'_{F,h}(x) 
			= x^k \sum_{d=1}^{\infty} \frac{(\mu*h)(d)N_F(d)}{d^k} + R^{(1)}(x) + R^{(2)}(x),
		\end{equation}
		where the series is absolutely convergent, and its sum is
		\begin{equation*}
			\sum_{d=1}^{\infty} \frac{(\mu*h)(d)}{d^k}= \frac1{\zeta(k)} \sum_{d=1}^{\infty} \frac{h(d)}{d^k}.
		\end{equation*}
		
		For the error terms in \eqref{sum_special} we have by Lemma \ref{Lemma_omega},
		\begin{equation*}
			R^{(1)}(x) \ll x^k \sum_{d>M(x)} \frac{|(\mu*h)(d)|}{d^k} \ll  
			x^k \sum_{d>M(x)} \frac{2^{\omega(d)}}{d^k} \ll x^k \frac{\log M(x)}{M(x)^{k-1}}\le
			x^k \frac{\log (Bx)}{(Ax)^{k-1}}\ll x \log x, 
		\end{equation*}
		\begin{equation*}
			R^{(2)}(x) \ll x^{k-1} \sum_{d\le M(x)} \frac{|(\mu*h)(d)|}{d^{k-1}}
			\ll x^{k-1} \sum_{d\le M(x)} \frac{2^{\omega(d)}}{d^{k-1}},
		\end{equation*}
		which is $\ll x^{k-1}$ for $k\ge 3$, and for $k=2$ it is $\ll x(\log M(x))^2\le x(\log Bx)^2\le x(\log x)^2$. 
	\end{proof}
	
	\begin{proof}[Proof of Corollary {\rm \ref{Corollary_linear}}]
		If $h=\delta$, then $\mu*h=\mu$. Hence $|(\mu*h)|\le 1$. It follows from the proof of Theorem \ref{Th_linear} that for $k=2$ the error is $x\log x$.
	\end{proof}
	
	{\bf Acknowledgement.} The author thanks Arnaud Bodin for helpful remarks and for pointing out Lemma \ref{Lemma_ineq}.

	\vskip2mm
	
	\noindent L\'aszl\'o T\'oth \\
	Department of Mathematics \\
	University of P\'ecs \\
	Ifj\'us\'ag \'utja 6, 7624 P\'ecs \\
	Hungary \\
	{\tt{ltoth@gamma.ttk.pte.hu}}
	
\end{document}